\documentclass{amsart}
\usepackage{amsmath,amssymb,amscd,latexsym,verbatim}

\usepackage{enumerate}
\usepackage[all]{xy}
\SelectTips{eu}{}
\CompileMatrices

\newcommand{\sst}{\scriptstyle}

\newcommand{\ges}{{\sst\geqslant}}
\newcommand{\les}{{\sst\leqslant}}

\newcommand{\ann}{\operatorname{Ann}}

\newcommand{\col}{\colon}
\newcommand{\dd}{\partial}
\newcommand{\depth}{\operatorname{depth}}
\newcommand{\grade}{\operatorname{grade}}
\newcommand{\EH}[4]{{}^{#1}\!\operatorname{E}_{#2,#3,#4}}
\newcommand{\ED}[4]{{}^{#1}\!\operatorname{d}_{#2,#3,#4}}

\newcommand{\reg}{\operatorname{reg}}
\newcommand{\slope}{\operatorname{slope}}
\newcommand{\rate}{\operatorname{rate}}
\newcommand{\Rate}{\operatorname{Rate}}

\newcommand{\hh}[1]{\operatorname{H}(#1)}
\newcommand{\HH}[2]{\operatorname{H}_{#1}(#2)}

\newcommand{\hra}{\hookrightarrow}

\newcommand{\irr}{\!\scriptscriptstyle{+}}
\newcommand{\Ker}{\operatorname{Ker}}

\newcommand{\lra}{\longrightarrow}

\newcommand{\xra}{\xrightarrow}

\newcommand{\pd}{\operatorname{pd}}
\newcommand{\rank}{\operatorname{rank}}

\newcommand{\Ext}[4]{\operatorname{Ext}^{#1}_{#2}(#3,#4){}}
\newcommand{\Hom}[3]{\operatorname{Hom}_{#1}(#2,#3)}
\newcommand{\Tor}[4]{\operatorname{Tor}_{#1}^{#2}(#3,#4){}}
\newcommand{\topp}[3]{t_{#1}^{#2}(#3){}}

\newcommand{\vf}{{\varphi}}
\newcommand{\wt}{\widetilde}
\newcommand{\tra}{\twoheadrightarrow}

\newcommand{\ini}{\operatorname{in}}

\newcommand{\BN}{{\mathbb N}}
\newcommand{\BR}{{\mathbb R}}
\newcommand{\BZ}{{\mathbb Z}}

\theoremstyle{plain}

\newtheorem{theorem}{Theorem}[section]
\newtheorem{proposition}[theorem]{Proposition}

\newtheorem{corollary}[theorem]{Corollary}
\newtheorem*{mtheorem}{Main Theorem}

\theoremstyle{definition}

\newtheorem{example}[theorem]{Example}
\newtheorem{chunk}[theorem]{}

\theoremstyle{remark}

\newtheorem{question}[theorem]{Question}

\numberwithin{equation}{theorem}

\begin{document}

\title[Free resolutions over Koszul algebras]
{Free resolutions over commutative \\ Koszul algebras}
\author[L.~L.~Avramov]{Luchezar L.~Avramov}
\address{Department of Mathematics,
  University of Nebraska, Lincoln, NE 68588, U.S.A.}
\email{avramov@math.unl.edu}
\author[A.~Conca]{Aldo~Conca}
\address{Dipartimento di Matematica, 
Universit«a di Genova, Via Dodecaneso 35, 
I-16146 Genova, Italy}
\email{conca@dima.unige.it}
\author[S.~B.~Iyengar]{Srikanth B.~Iyengar}
\address{Department of Mathematics,
  University of Nebraska, Lincoln, NE 68588, U.S.A.}
\email{iyengar@math.unl.edu}
\date{\today}
\thanks{Research partly supported by NSF grants DMS 0803082 (LLA)
and 0602498 (SBI)}



\date{\today}

 \begin{abstract} 
For $R=Q/J$ with $Q$ a commutative graded algebra over a field and $J\ne0$, 
we relate the slopes of the minimal resolutions of $R$ over $Q$
and of $k=R/R_{+}$ over $R$.  When $Q$ and $R$ are Koszul and $J_1=0$
we prove $\Tor iQ{R}k_j=0$ for $j>2i\ge0$, and also for $j=2i$ when
$i>\dim Q-\dim R$ and $\pd_QR$ is finite.
 \end{abstract}

\maketitle

Let $K$ be a field and $Q$ a commutative $\BN$-graded $K$-algebra with
$Q_0=K$.  Each graded $Q$-module $M$ with $M_j=0$ for $j\ll0$ has a unique
up to isomorphism minimal graded free resolution, $F^M$.  The module
$F^M_i$ has a basis element in degree $j$ if and only if $\Tor iQkM_j\ne0$
holds,  where $k=Q/Q_{\irr}$ for $Q_{\irr}=\bigoplus_{j\ges1}Q_j$.  Important 
structural information on $F^M$ is encoded in the sequence of numbers
  \[
\topp iQM=\sup\{j\in\BZ\mid\Tor iQkM_j\ne0\}\,.
  \]
It is distilled through the notion of \emph{Castelnuovo-Mumford
regularity}, defined by
  \[
\reg_QM=\sup_{i\ges 0}\{\topp iQM-i\}\,.
  \]
One has $\reg_Qk\ge0$, and equality means that $Q$ is \emph{Koszul};
see, for instance, \cite{PP}.

When the $K$-algebra $Q$ is finitely generated, every finitely genetrated
graded $Q$-module $M$ has finite regularity if and only if $Q$ is a
polynomial ring over some Koszul algebra, see~\cite{AP}; by contrast,
the \emph{slope} of $M$ over $Q$, defined to be the real number
 \[
\slope_{Q}M=\sup_{i\ges 1}\left\{\frac{\topp iQM-\topp0QM}{i}\right\}\,,
 \]
is always finite; see Corollary \ref{cor:rate}.   Following Backelin \cite{Ba},
we set $\Rate Q=\slope_QQ_{\irr}$ and note that one has $\Rate Q\geq 1$,
with equality if and only if $Q$ is Koszul.

\begin{mtheorem}
  \label{thm:main} 
If $Q$ is a finitely generated commutative Koszul $K$-algebra and $J$ a 
homogeneous ideal with $0\ne J\subseteq(Q_{\irr})^2$,  then for $R=Q/J$ 
and $c=\Rate R$ one has
  \begin{enumerate}[\rm(1)]
  \item
$\max\{c,2\}\le\slope_QR\le c+1$, with $c<\slope_QR$ when $\pd_QR$ is finite.
  \item
$\topp iQR=(c+1)\cdot i$ for some $i\ge1$ implies the following conditions:
\newline $\topp hQR=(c+1)\cdot h$ for
$1\le h\le i$ and $i\le\rank_k(J/Q_{\irr}J)_{c+1}$.
  \item
$\topp iQR<(c+1)\cdot i$ holds for all $i>\dim Q-\dim R$ when
$\pd_QR$ is finite.
  \item
$\reg_QR\le c\cdot\pd_QR$; when $Q$ is a standard graded polynomial 
ring, equality holds if and only if $J$ is generated by a $Q$-regular 
sequence of forms of degree $c+1$.
  \end{enumerate}
   \end{mtheorem}

The result is new even in the case of a polynomial ring $Q$, where 
a related statement was initially proved by using Gr\"obner bases;
see \ref{rem:taylor}.

The theorem is proved in Section \ref{sec:koszul}.  Its assertions
have very different underpinnings:  The inequalities in (1) come from
results in homological algebra, established in Section \ref{sec:rate}
with no finiteness or hypotheses on~$Q$.  The remaining statements
are deduced from results about small homomorphism $Q\to R$,  proved
in Section \ref{sec:small} by using delicate properties of commutative
noetherian rings.

Much of the discussion in the body of the paper concerns
the general problem of relating properties of the numbers $\slope_QM$,
$\slope_QR$, and $\slope_RM$, when $Q\to R$ is a homomorphism of graded
$K$-algebras and $M$ is a graded module defined over~$R$.

The essence of our results is a comparison of two types of degrees, ones 
arising from homological considerations, the others induced by internal 
gradings of the objects under study.   In constructions involving two or 
more gradings the index referring to an internal degree always
appears last.  When $y$ is a homogeneous
element of a bigraded object, $|y|$ denotes the \emph{homological
degree} and $\deg(y)$ the \emph{internal degree}.  

The proofs presented below involve various homological constructions that
are well documented in the case of commutative local rings and their
local homomorphisms, but for which graded analogs may be difficult to
find in the literature.  When explicit information on the behavior of
internal degrees is needed, we give the statements in the graded context
with references to sources dealing with the local situation.  We have
verified---and invite readers to follow suit---that in these instances
an internal degree can be factored in all the arguments involved.

\section{Slopes of graded modules}
  \label{sec:rate}

In this section $\vf\col Q\to R$ is a surjective homomorphism of 
graded $K$-algebras, and $M$ is a graded $R$-module with 
$M_j=0$ for all $j\ll0$; we set $J=\Ker\vf$. 

We recall a classical change-of-rings spectral sequence of Cartan and
Eilenberg. 

  \begin{chunk}
    \label{ce}
By \cite[Ch.\,XVI, \S5]{CE}, there exists a spectral sequence of trigraded $k$-vector 
spaces
  \begin{equation}
    \label{eq:cesequence}
\EH rpqj\underset{p}{\implies}\Tor{p+q}QkM_j \quad\text{for}\quad r\ge2\,,
  \end{equation}
with differentials acting according to the pattern
  \begin{equation}
    \label{eq:cedifferential}
\ED rpqj\col \EH rpqj\to\EH r{p-r}{q+r-1}j \quad\text{for}\quad r\ge2\,,
  \end{equation}
with second page of the form
  \begin{equation}
    \label{eq:ceE2}
\EH 2pqj\cong \bigoplus_{j_1+j_2=j}\Tor pRkM_{j_1}\otimes_{k}\Tor qQkR_{j_2}\,,
  \end{equation}
and with edge homomorphisms
  \begin{equation}
    \label{eq:ceedge}
\Tor iQkM_{j}\tra\EH{\infty}i0j=\EH{i+1}i0j\hra\EH 2i0j\cong\Tor iRkM_{j}
  \end{equation}
equal to the canonical homomorphisms of $k$-vector spaces
  \begin{equation}
    \label{eq:cechange}
\Tor i{\vf}kM_j\col\Tor iQkM_j\to\Tor iRkM_j\,.
  \end{equation}
  \end{chunk}

For all $r$, $p$, and $q$ we set $\sup\EH rpq*=\sup\{j\in\BZ\mid \EH rpqj\ne0\}$.

The proof of the next result is based on an analysis of the convergence
of the preceding change-of-rings spectral sequence on the line $q=0$.

\begin{proposition}
\label{thm:ceub}
When $J\ne QJ_1$ holds there are inequalities
\[
\slope_{R}M \leq \max\left\{\slope_{Q}M\,, 
     \sup_{i\ges 1}\left\{\frac{\topp iQR-1}{i}\right\}\right\}
\leq\max\{\slope_{Q}M,\slope_QR\}\,.
\]
\end{proposition}

 \begin{proof}
If $\topp iQR$ or $\topp  iQM$ is infinite for some $i\ge0$, then so are
both maxima above, hence there is nothing to prove.  Thus, we may assume
that $\topp iQR$ and $\topp  iQM$ are finite for every $i\ge0$; in this
case the second inequality is clear.  Let $m$ denote the middle term in 
the inequalities above.  Using the equality $\topp 0QM=\topp 0RM$, 
we get
  \begin{align}
   \tag*{(\ref{thm:ceub}.1)${}_{i}$}
\topp iQM&\leq mi+\topp 0RM\,;
\\
   \tag*{(\ref{thm:ceub}.2)${}_{i}$}
\topp iQR&\leq mi+1\,.
  \end{align}

For $i\ge0$ and $r\geq 2$, from fomulas \eqref{eq:cedifferential} and
\eqref{eq:ceE2} one gets exact sequences
  \begin{equation}
    \label{eq:celimit}
0\lra\EH{r+1}i0j\lra\EH ri0j\xra{\ \ED ri0j \ }\EH r{i-r}{r-1}j \,.
  \tag{\ref{thm:ceub}.3}
  \end{equation}
We set up a primary induction on $i$ and a secondary, descending one,
on $r$ to prove 
 \begin{align}
   \tag*{(\ref{thm:ceub}.4)${}_{i,r}$}
\sup \EH ri0* &\le mi + \topp 0RM
\quad\text{and}\quad i+1\ge r\ge2\,.
 \end{align}
In view of \eqref{eq:ceE2}, the validity of {(\ref{thm:ceub}.4)${}_{i,2}$}
is the assertion of the proposition.

The basis of the primary induction, for $i=1$, comes from
\eqref{eq:ceedge} and (\ref{thm:ceub}.1)${}_{1}$.

Fix an integer $i\ge2$ and assume that (\ref{thm:ceub}.4)${}_{i',r}$ holds
for $i'<i$.  Formulas \eqref{eq:ceedge} and (\ref{thm:ceub}.1)${}_{i}$
imply (\ref{thm:ceub}.4)${}_{i,i+1}$.  Fix $r\in[2,i]$ and assume that
(\ref{thm:ceub}.4)${}_{i,r'}$ holds for $i+1\ge r'>r$. The first relation
in the following chain
  \begin{align*}
\sup \EH{r}i0* 
  &\leq \max\{\sup \EH {r+1}i0* \,, \sup \EH {r}{i-r}{r-1}* \}\\
  &\leq \max\{mi+\topp 0RM \,, \sup \EH {r}{i-r}{r-1}* \}\\
  &\leq \max\{mi+\topp 0RM \,, \sup \EH {2}{i-r}{r-1}* \}\\
  &= \max\{mi+\topp 0RM \,, \topp{i-r}RM+\topp{r-1}QR \}\\
  &\leq \max\{mi+\topp 0RM \,,(m(i-r)+\topp 0RM) + (m(r-1)+1)\}\\
  & = \max\{mi+\topp 0RM \,, mi+\topp 0RM-(m-1)\}\\
  &\leq mi + \topp 0RM
    \end{align*}
comes from the exact sequence \eqref{eq:celimit}.  The second one holds
by (\ref{thm:ceub}.4)${}_{i,r+1}$, the third because $\EH{r}{i-r}{r-1}*$
is a subfactor of $\EH 2{i-r}{r-1}*$, the fourth by \eqref{eq:ceE2}, the
fifth by (\ref{thm:ceub}.4)${}_{i-r,2}$ and (\ref{thm:ceub}.2)${}_{r-1}$,
and the last one because $J\ne QJ_1$ implies $m\geq1$.

This completes the inductive proof of the inequality (\ref{thm:ceub}.4)${}_{i,r}$.
  \end{proof}

Variants of the proposition have been known for some time, at
least when $M$ is finitely generated and $R$ is \emph{standard 
graded}; that is, $R=K[R_1]$ with $\rank_KR_1$ finite.  Thus, 
Aramova, B\u arc\u anescu, and Herzog in \cite[1.3]{ABH} 
established the corresponding result for a related invariant, 
$\rate_RM=\sup_{i\ges1}\{\topp iQM/i\}$.  They used the same 
spectral sequence, extending an argument of Avramov for
$M=k$, see \cite[p.~97]{Ba}; in the latter case, the corollary below
was first proved by Anick in \cite[4.2]{An}. 

  \begin{corollary}
  \label{cor:rate} 
If $R$ is finitely generated over $K$, then for every finitely
generated $R$-module $M$ one has $\slope_RM<\infty$.
  \end{corollary}

  \begin{proof}
One may choose $Q$ to be a polynomial ring in finitely many indeterminates
over $K$.  In this case $\Tor iQkR_*$ and $\Tor iQkM_*$ are finitely generated 
over $k$ for each $i\ge0$ and are zero for almost all $i$, so $\slope_QR$ and
$\slope_QM$ are finite. 
  \end{proof}

In the proof of the next result we again use the spectral sequence
in \ref{ce}, this time analyzing its convergence on the line $p=0$.
The hypothesis includes a condition on the maps $\Tor i{\vf}kM_j$;
see \ref{ch:small} and Proposition \ref{prop:koszul_small} for 
situations where it is met.

\begin{proposition}
\label{thm:celb}
If $M\ne0$ and $\Tor i{\vf}kM$ is injective for each $i$, then one has
  \begin{align*}
 \slope_{Q}R &\leq 1+ s
 \quad\text{where}\quad
s=\sup_{i\ges 2}\left\{\frac{\topp iRM - \topp 0RM-1}{i-1}\right\}\,.
  \end{align*}
 \end{proposition}

\begin{proof}
The hypothesis implies $\topp 0RM>-\infty$. There is nothing 
to prove if $\topp iQM=\infty$ for some $i$, so we assume that 
$\topp  iQM$ is finite for all $i\ge0$.   By the definition of the number 
$s$, the following inequalities then hold: 
  \begin{equation}
    \label{eq:sup}
     \tag*{(\ref{thm:celb}.1)${}_{i}$}
\topp iRM \leq s(i-1)+1+\topp 0RM\quad \text{for all}\quad i\ge 2\,.
  \end{equation}

It follows from \eqref{eq:cedifferential} and \eqref{eq:ceE2} that for 
$r\geq 2$ there exist exact sequences
  \begin{equation} \label{eq:cecolimit}
\EH rr{i-r+1}j\xra{\ \ED rr{i-r+1}j \ }\EH r0ij\lra\EH {r+1}0ij\lra0
    \tag*{(\ref{thm:celb}.2)}
  \end{equation}

By primary induction on $i$ and secondary, descending induction
on $r$, we prove
  \begin{align}
     \tag*{(\ref{thm:celb}.3)${}_{i,r}$}
\sup \EH r0i*& \leq (s+1)i + \topp 0RM
\quad\text{for}\quad i+2\ge r\ge2\,.
  \end{align}
In view of \eqref{eq:ceE2}, the validity of {(\ref{thm:celb}.3)${}_{i,2}$}
yields the assertion of the proposition.

The injectivity of $\Tor{}{\vf}kM$ and \eqref{eq:ceedge} imply
$\EH{\infty}pq*=0$ for $q\ge 1$ and all $p$.  It follows from \eqref{eq:cedifferential}
and \eqref{eq:ceE2} that $\EH {n+2}0i*$ is isomorphic to $\Tor 0RkM_*$
for $i=0$ and to $0$ for $i\ge1$, so {(\ref{thm:celb}.3)${}_{i,i+2}$}
holds for all $i\ge0$.  This gives the basis of the primary induction
for $i=0$ and that of the secondary induction for all $i\ge1$.

Fix an integer $i\ge1$ and assume that (\ref{thm:celb}.3)${}_{i',r'}$
holds for all pairs $(i',r')$ with $i'<i$ and $i+2\ge r'>r$.  One then
has a chain of relations
  \begin{align*}
\sup \EH rr{i-r+1}* 
   &\leq \sup \EH 2r{i-r+1}* \\
   & = \topp rRM + \topp{i-r+1}QR\\
   &\leq \topp rRM +  (s+1)(i-r+1)\\
   &\leq s(r-1)+1 + \topp 0RM +  (s+1)(i-r+1) \\
   &= (s+1)i + (2-r) + \topp 0RM\\
   &\leq (s+1)i + \topp 0RM\,,
\end{align*}
where the first one holds because $\EH rr{i-r+1}* $ is a subfactor
of $\EH 2r{i-r+1}*$, the second by formula \eqref{eq:ceE2},
the third by (\ref{thm:celb}.3)${}_{i-r+2,2}$ and \eqref{eq:ceE2}, and the fourth by
(\ref{thm:celb}.1)${}_{r}$.  The exact sequence \ref{eq:cecolimit},
the preceding inequalities, and (\ref{thm:celb}.3)${}_{i,r+1}$ give
 \begin{align*}
\sup \EH r0i*
   &\leq\max\{\sup\EH{r+1}0i*\,,\sup\EH rr{i-r+1}*\}
  \\ &\leq (s+1)i + \topp 0RM\,.
 \end{align*}

Hereby, the inductive proof of the inequality (\ref{thm:celb}.3)${}_{i,r}$ 
is complete.
  \end{proof}

\section{Regular elements}
  \label{sec:reg}

Not surprisingly, the bounds obtained in the preceding section can be 
sharpened in cases when the minimal free resolution of $R$ or of $M$ 
over $Q$ is particularly simple.  

In this section we discuss a classical avatar of this phenomenon.

  \begin{proposition}
 \label{thm:reg}
If $R=Q/(f)$ for a non-zero divisor $f\in Q_{\irr}$, then one has: 
   \begin{alignat}{3}
     \tag{1}
\slope_QM&\le\max\{\slope_RM,\deg(f)\}
&\quad&\text{with equality for } \quad &f&\notin (Q_{\irr})^2\,.
  \\
     \tag{2}
\slope_RM&\le\max\{\slope_QM,\deg(f)/2\}
&\quad&\text{with equality for} \quad &f&\in Q_{\irr}\ann_QM\,.
   \end{alignat}
 \end{proposition}

\begin{proof}
We start by noting an elementary inequality that will be invoked
a couple of times:  All pairs of real numbers $(a_1,a_2)$ and $(b_1,b_2)$
with positive $b_1$ and $b_2$ satisfy
  \begin{equation}
   \label{eq:short}
\frac {a_1+a_2}{b_1+b_2} \leq 
\max\left\{\frac{a_1}{b_1}\,,\,\frac{a_2}{b_2}\right\}\,.
    \end{equation}

Set $d=\deg(f)$.  The minimal graded free resolution of $R$ over $Q$ is
  \[
0\lra Q(-d) \xra{\ f \ } Q\lra 0
  \]
so $\Tor qQRk$ vanishes for $q\ne0,1$, is isomorphic to $k$ for 
$q=0$, and to $k(-d)$ for $q=1$, so for each pair $(i,j)$ the 
spectral sequence \ref{ce} yields an exact sequence
  \begin{equation}
   \label{eq:long}
 \begin{gathered}
\xymatrixcolsep{1.3pc}
\xymatrixrowsep{.3pc}
\xymatrix {
&&\Tor{i+1}RkM_{j}\ar@{->}[rr]^-{\delta_{i+1,j}}
&&\Tor{i-1}RkM_{j-d}
  \\
\ar@{->}[r]
&\Tor{i}QkM_{j}\ar@{->}[r]
&\Tor{i}RkM_{j}\ar@{->}[rr]^-{\delta_{i,j}}
&&\Tor{i-2}RkM_{j-d}
}
\end{gathered}
    \end{equation}
The one for $i=0$ gives the following equality:
  \begin{equation}
   \label{eq:zero}
\topp{0}QM=\topp{0}RM\,.
    \end{equation}

(1) For $i\ge1$ the middle three terms of the exact sequences \eqref{eq:long} yield
  \begin{equation}
   \label{eq:long1}
    \begin{aligned}
\topp iQM
&\le\max\{\topp iRM,(\topp{i-1}RM+d)\}
  \end{aligned}
    \end{equation}
{From} \eqref{eq:long1}, \eqref{eq:zero},
and \eqref{eq:short} we obtain the inequalities below:
  \begin{align*}
\slope_QM
&=\sup_{i\ges1}\left\{\frac{\topp{i}QM-\topp 0QM}{i}\right\}
  \\
&\le \sup_{i\ges1}\left\{\max\left\{\frac{\topp{i}RM-\topp 0RM}i\,,\,
\frac{(\topp{i-1}RM-\topp 0RM)+d}{(i-1)+1}\right\}\right\}
  \\
&\leq\sup_{i\ges2}\left\{\max\left\{\frac{\topp{i}RM-\topp 0RM}{i}\,,
\frac{\topp{i-1}RM-\topp 0RM}{i-1},d\right\}\right\}
  \\
&=\max\left\{\sup_{i\ges1}
\left\{\frac{\topp{i}RM-\topp 0RM}{i}\right\},d\right\}
  \\
&=\max\left\{\slope_RM,d\right\}\,.
  \end{align*}

When $f\notin(Q_{\irr})^2$ holds, the proof in \cite[3.3.3(1)]{Av:barca} of a 
result of Nagata, implies $\delta_{i,j}=0$ in \eqref{eq:long}, so equalities 
hold in \eqref{eq:long1}.  This and \eqref{eq:zero} give
  \begin{align*}
\topp{1}QM-\topp{0}QM&=\max\{\topp{1}RM-\topp{0}RM,d\}\,,
  \\
\topp{i}QM-\topp{0}QM&\ge\topp{i}RM-\topp{0}RM
\quad\text{for}\quad i\ge2\,.
  \end{align*}
The preceding relations clearly imply $\slope_QM\ge\max\{\slope_RM,d\}$.
  
(2)  For $i\ge1$ the last three terms of the exact sequences \eqref{eq:long} yield
  \begin{equation}
   \label{eq:long2}
    \begin{aligned}
\topp iRM
&\le\max\{\topp iQM,(\topp{i-2}RM+d)\}\\
&\le\max\{\topp iQM,(\topp{i-2}QM+d),(\topp{i-4}RM+2d)\}
\le\cdots\\
&\le\max_{0\les 2h\les i}\{\topp{i-2h}QM+hd\}\,.
  \end{aligned}
    \end{equation}

{From} \eqref{eq:long2}, \eqref{eq:zero},
and \eqref{eq:short} we obtain the inequalities below:
  \begin{align*}
\slope_RM
&=\sup_{i\ges1}\left\{\frac{\topp{i}RM-\topp 0RM}{i}\right\}
  \\
&\le\sup_{i\ges1}\left\{\max_{0\les 2h\les i}
\left\{\frac{\topp{i-2h}QM-\topp 0QM+hd}{(i-2h)+(2h)}\right\}\right\}
  \\
&\leq\sup_{i\ges1}\left\{\max_{0\les 2h< i}
\left\{\frac{\topp{i-2h}QM-\topp 0QM}{i-2h}\,,\,\frac{d}{2}\right\}\right\}
  \\
&=\max\left\{\sup_{i\ges1}
\left\{\frac{\topp{i}QM-\topp 0QM}{i}\right\}\,,\,\frac{d}{2}\right\}
  \\
&=\max\left\{\slope_QM\,,\,\frac{d}{2}\right\}\,.
  \end{align*}

For $f\in Q_{\irr}\ann_QM$, the proof in \cite[3.3.3(2)]{Av:barca} of a 
result of Shamash shows that $\delta_{i,*}$ in \eqref{eq:long} is surjective, 
so equalities hold in \eqref{eq:long2}; in view of \eqref{eq:zero} one gets
  \begin{align*}
\topp{1}RM-\topp{0}RM&=\topp{1}QM-\topp{0}QM\,,
  \\
\topp{2}RM-\topp{0}RM&=\max\{\topp{2}QM-\topp{0}QM,d\}\,,
  \\
\topp{i}RM-\topp{0}RM&\ge\topp{i}QM-\topp{0}QM \quad\text{for}\quad
i\ge3\,.
  \end{align*}
These relations clearly imply an inequality
$\slope_RM\ge\max\{\slope_QM,d/2\}$.
   \end{proof}

\section{Small homomorphisms of graded algebras}
  \label{sec:small}

A homomorphism $\vf\col Q\to R$ of graded $K$-algebras is called
\emph{small} if the map
\[
\Tor i{\vf}kk_j\col \Tor iQkk_j \to \Tor iRkk_j 
\]
is injective for each pair $(i,j)\in\BN\times\BZ$; see \ref{ch:small} for 
examples.  Recall that \emph{homological products} turn $\Tor{}QkR$ 
into a bigraded algebra; see \cite[Ch.\,XI, \S4]{CE}.

\begin{theorem}
 \label{thm:small} 
Let $Q$ be a standard graded $K$-algebra, $\vf\col Q\to R$ a surjective
small homomorphism of graded $K$-algebras with $\Ker\vf\ne0$, and set
$c=\Rate R$.

For every integer $i\ge1$ there are then inequalities
  \[
{\topp iQR} \le \slope_QR\cdot i \le (c+1)\cdot i\,,
  \]
and the following conditions are equivalent:
  \begin{enumerate}[\quad\rm(i)]
 \item
$\topp iQR=(c+1)\cdot i$.
 \item
$\topp hQR=(c+1)\cdot h$ for $1\le h\le i$.
 \item
$\topp1QR=c+1$ and $\Tor iQkR_{i(c+1)}=(\Tor 1QkR_{c+1})^i\ne0$.
  \end{enumerate}
   \end{theorem}

Before starting on the proof of the theorem we present an application, followed 
by a couple of easily verifiable sufficient conditions for the smallness of $\vf$.
 
 \begin{corollary}
 \label{cor:small} 
With $J=\Ker\vf$, the following assertions hold:
  \begin{enumerate}[\rm(1)]
  \item
$\topp iQR=(c+1)\cdot i$ for some $i\ge1$ implies the conditions \newline
$\topp hQR=(c+1)\cdot h$ for $1\le h\le i$ and $i\le\rank_k(J/Q_{\irr}J)_{c+1}$.
  \item
$\topp iQR<(c+1)\cdot i$ holds for all $i>\dim Q-\dim R$ when 
$\pd_QR$ is  finite.
  \item
$\reg_QR\le c\cdot\pd_QR$.
  \end{enumerate}
    \end{corollary}

  \begin{proof}
Homological products are strictly skew-commutative for the homological
degree, see \cite[Ch.\,XI, \S4]{CE}, so $(\Tor 1QkR_*){}^i$ is the image of 
a canonical $k$-linear map
  \[
\lambda_{i,*}\col \textstyle{\bigwedge}^i_k(J/Q_{\irr}J)_*\cong
\textstyle{\bigwedge}^i_k \Tor 1QkR_* \to \Tor iQkR_*\,.
  \]

(1)  This follows from the map above and the implication (i)$\implies$(ii) and (iii).

(2)  When $\pd_QR$ is finite one has $\grade_QR=\dim Q-\dim R$ by a 
theorem of Peskine and Szpiro~\cite{PS}, and $\lambda_{i,*}=0$ for 
$i>\grade_QR$ from a theorem of Bruns~\cite{Br}.  Thus, Theorem 
\ref{thm:small} implies ${\Tor iQkR}_j=0$ for $j\ge(c+1)i$.

(3) The theorem gives $\topp iQR-i\le ci$ for each $i$, hence
$\reg_QR\le c\cdot\pd_QR$.
   \end{proof}

A bit of notation comes in handy at this point.

  \begin{chunk}
    \label{canonical}
A standard graded $K$-algebra $R$ has a \emph{canonical presentation}
$R=\wt R/I_R$ with $\wt R$ the symmetric $K$ algebra on $R_1$ and 
$I_R\subseteq(\wt R_{\irr})^2$, obtained from the epimorphism 
of $K$-algebras $\wt R\to R$ extending the identity map on $R_1$.

If $Q$ is standard graded $K$-algebra and $\vf\col Q\to R$ is a surjective
homomorphism with $\Ker\vf\subseteq(Q_{\irr})^2$, then $\wt R\to R$ 
factors as ${\wt R}\cong\wt Q\to Q\xra{\vf}R$.   
  \end{chunk}
  
 \begin{chunk}
   \label{ch:small}
A homomorphism $\vf$ as on \ref{canonical} is small if $J=\Ker\vf$ satisfies 
one of the conditions:
  \begin{enumerate}[\quad\rm(a)]
 \item
$J\subseteq(f_1,\dots,f_a)$, where $f_1,\dots,f_a$ is some
$Q$-regular sequence in $Q_{\irr}$.
 \item
$J_j=0$ for $j\le\reg_{\wt Q}Q$, where $Q={\wt Q}/I_Q$ is the 
canonical presentation.
  \end{enumerate}

Indeed, see \cite[4.3]{Av:small} for (a), and \c Sega \cite[5.1, 9.2(2)]{Se} 
for (b). 
 \end{chunk}

\emph{The hypothesis of Theorem \emph{\ref{thm:small}} are in force 
for the rest of this section.}  The proof of the theorem utilizes free 
resolutions with additional structure.  

A \emph{model} of $\vf$ is a differential bigraded $Q$-algebra
$Q[X]$ with the following properties:  For $n\ge1$ here exist
linearly independent over $K$ homogeneous subsets $X_n=\{x\in X\mid
|x|=n\}$, such that the underlying bigraded algebra is isomorphic
to $Q\otimes_K\bigotimes_{n=1}^{\infty}K[X_n]$, where $K[X_n]$ is
the exterior algebra of the graded $K$-vector space $KX_n$ when $n$
is odd, and the symmetric algebra of that space when $n$ is even.
The differential satisfies $\deg(\dd(y))=\deg(y)$ for every element $y\in
Q[X]$, and the following sequence of homomorphisms of free graded
$Q$-modules is resolution of $R$:
  \[
\cdots \lra Q[X]_{n,*}\xra{\,\dd\,} Q[X]_{n-1,*}\cdots \lra\cdots
\lra Q[X]_{0,*}\lra 0
  \]

A $Q$-basis of $Q[X]$ is provided by the set consisting of $1$ and
all the monomials $x_{1}^{d_1}\cdots x_{s}^{d_s}$ with $x_r\in X$, and
with $d_r=1$ when $|x_r|$ is odd, respectively, $d_r\ge1$ when $|x_r|$
is even.  The model $Q[X]$ is said to be \emph{minimal} if for each $x\in
X$,  the coefficient of every $x_i\in X$ in the expansion of $\dd(x)$
is contained in $Q_{\irr}$.

We summarize the properties of minimal models used in our arguments.

  \begin{chunk}
    \label{model:exist}
A minimal model $Q[X]$ of $\vf$ always exists, and is unique up to
non-canonical isomorphism of differential bigraded $Q$-algebras; see
\cite[7.2.4]{Av:barca}.  In such a model $\dd(X_1)$ is a minimal set of
homogeneous generators of the $\Ker\vf$ and $Q[X_1]$ is the Koszul 
complex on that set, with its standard bigrading, differential and 
multiplication.
 \end{chunk}

  \begin{chunk}
    \label{model:omega}
Let ${\wt R}[Z]$ be a minimal model for the canonical presentation 
${\wt R}\to R$, see \ref{canonical}.  Let $Z_0$ be a $K$-basis of 
${\wt R}_1$, and choose a $k$--linearly independent set
  \[
Z'=\{z'\mid |z'|=|z|+1\text{ and }\deg(z')=\deg(z)\}_{z\in Z_0\sqcup Z}\,.
  \]
By \cite[7.2.6]{Av:barca}, there exists an isomorphism of bigraded 
$k$-vector spaces
  \[
\Tor{}Rkk\cong\bigotimes_{n=1}^{\infty}k\langle Z'_n\rangle\,,
  \]
where $k\langle Z'_n\rangle$ denotes the exterior algebra of the 
graded $k$-vector space $kZ'_n$ when $n$ is odd, and the divided 
powers algebra of that space when $n$ is even.
   \end{chunk}

  \begin{chunk}
    \label{model:small}
Let $Q[X]$ be a minimal model for $\vf$, and let ${\wt R}\xra{\psi}Q\xra{\vf}R$
be a factorization of the canonical presentation $\wt R\to R$ as in \ref{canonical}.  
If ${\wt R}[Y]$ is a minimal model for $\psi$, then there is a minimal model 
${\wt R}[Z]$ of $\wt R\to R$ with $Z=Y\sqcup X$; see \cite[4.11]{AI}.
  \end{chunk}

  \begin{chunk}
\begin{proof}[Proof of Theorem \emph{\ref{thm:small}}]
For every integer $i\ge2$ the following equality holds:
  \begin{equation}
    \label{eq:c}
\topp{i-1}R{R_{\irr}}-\topp{0}R{R_{\irr}}=\topp{i}Rk-1\,.
  \end{equation}
Thus, for $i\ge1$ the definition of slope and Proposition \ref{thm:celb} 
applied with $M=k$ give
  \begin{equation}
   \label{eq:t}
{\topp iQR} /i\le\slope_QR\le c+1\,.
  \end{equation}

It remains to establish the equivalence of the conditions in the theorem.

(iii)$\implies$(ii).  
The condition $(\Tor 1QkR_{c+1})^i\ne0$ forces $(\Tor 1QkR_{c+1})^h\ne0$
for $h=1,\dots,i$.  As $\Tor{}QkR$ is a bigraded algebra, one gets
 \[
\Tor hQkR_{(c+1)h}\supseteq(\Tor 1QkR_{c+1})^h\ne0\,.
 \]
This implies $\topp hQR\ge(c+1)h$, and \eqref{eq:t} provides the 
converse inequality.

(ii)$\implies$(i).  This implication is a tautology.

(i)$\implies$(iii).  
The hypothesis means $\Tor iQkR_{i(c+1)}\ne0$, so
we have to prove
  \begin{equation}
    \label{eq:kx0}
\Tor iQkR_{i(c+1)}=(\Tor 1QkR_{c+1})^i\,.
  \end{equation}

Let $Q[X]\to R$ be a minimal model and set $k[X]=k\otimes_QQ[X]$.
The bigraded $k$-algebras $\hh{k[X]}$ and $\Tor{}QkR$ are isomorphic,
with  
  \begin{equation}
\Tor iQkR_j\cong\HH i{k[X]}_j\,.
    \label{eq:hkx}
  \end{equation}

In view of \ref{model:small} each $x\in X_n$ can be viewed as an
indeterminate of a minimal model of ${\wt R}\to R$, and so by \ref{model:omega}
it defines an element $x'$ in $\Tor{n+1}Rkk$ with $\deg(x) =\deg(x')$.
{From} this equality and \eqref{eq:c} we obtain
  \begin{equation}
    \label{eq:x}
\deg(x)
 =\deg(x') \le\topp{n+1}Rk \le cn+1=c|x|+1\,.
   \end{equation}
The $k$-vector space $k[X]_{i,(c+1)i}$ has a basis of monomials
$x_1^{d_1}\cdots x_s^{d_s}$ with $x_r\in X$ and $d_r\ge1$.  The following
relations hold, with the inequality coming from \eqref{eq:x}:
  \begin{align*}
\sum_{r=1}^sd_r|x_r| &=\big|x_1^{d_1}\cdots x_s^{d_s}\big|=i=(c+1)i-ci\\
&=\deg\big(x_1^{d_1}\cdots x_s^{d_s}\big)-c\big|x_1^{d_1}\cdots
x_s^{d_s}\big| 
  \\
&=\sum_{r=1}^sd_r(\deg(x_r)-c|x_r|)
  \\ 
&\le\sum_{r=1}^sd_r\,.
  \end{align*}
All $d_r$ and $|x_r|$ are positive integers, so for $1\le r\le s$ we get first
$|x_r|=1$, then $|x_r|=\deg(x_r)-c|x_r|$; thats is,  $\deg(x_r)=c+1$.  We 
have now proved
  \[
k[X]_{i,(c+1)i}=k[X_1]_{i,(c+1)i}=(kX_{1,c+1})^i\,.
  \]
The isomorphism \eqref{eq:hkx} maps $\Tor 1QkR_{c+1}$ to $kX_{1,c+1}$
and $\Tor iQkR_{(c+1)i}$ to a quotient of $k[X]_{i,(c+1)i}$, so the
equalities above establish \eqref{eq:kx0}.
  \end{proof}
 \end{chunk}

  \section{Koszul agebras}
    \label{sec:koszul}
    
In this section we prove and discuss the theorem stated in the introduction.

Here $Q$ is a standard graded $K$-algebra, $\vf\col Q\to R$ a 
surjective homomorphism of graded $K$-algebras, and $M$ a 
graded $R$-module.  As in \cite{PP}, we say that $M$ is 
\emph{Koszul} over $Q$ if $\Tor iQkM_j=0$ unless $i=j$.
In the following proposition the Koszul hypotheses are related 
to the injectivity of $\Tor{}{\vf}kM$ through the following lemma.

 \begin{proposition}
   \label{prop:koszul_small}
Assume that $J$ is contained in $(Q_{\irr})^2$.
  \begin{enumerate}[\rm(1)]
 \item
If $Q$ is Koszul, then $\vf$ is small.
 \item     
If $\vf$ is small and $M$ is Koszul over $Q$, then $\Tor{}{\vf}kM$
is injective.
  \end{enumerate}
 \end{proposition}

  \begin{proof}
Forming vector space duals, one sees that the injectivity of $\Tor{}{\vf}kM$ 
is equivalent to surjectivity of the homomorphism of bigraded $k$-vector spaces
  \[
\Ext{}{\vf}Mk\col\Ext{}{R}Mk\to\Ext{}{Q}Mk\,.
  \]

(1)  For $M=k$ the map above is a homomorphism of $K$-algebras, with
multiplication given by Yoneda products.  The map $\Ext1{\vf}kk_*$
is isomorphic to
  \[
\Hom R{\vf_1}k_*\col\Hom R{R_1}k_*\to\Hom Q{Q_1}k_*\,,
  \]
which is bijective as $J\subseteq(Q_{\irr})^2$ holds.  As $Q$ is
Koszul, the $k$ -algebra $\Ext{}{Q}kk$ is generated by $\Ext1{Q}kk$,
see \cite[Ch.\,2, \S1, Def.\,1]{PP}, so $\Ext{}{\vf}kk$ is surjective.

(2)  Yoneda products turn $\Ext{}{\vf}Mk$ into a homomorphism of
bigraded left modules over $\Ext{}{R}kk$, with this algebra acting
on $\Ext{0}{Q}Mk$ through $\Ext{}{\vf}kk$.   The bigraded module
$\Ext{}{Q}Mk$ is generated over $\Ext{}{Q}kk$ by $\Ext{0}{Q}Mk$, because
$M$ is Koszul over $Q$; see \cite[Ch.\,2, \S1, Def.\,2]{PP}.  Since $\vf$
is small, $\Ext0{\vf}kk_*$ is surjective, and hence $\Ext0{Q}Mk$
generates $\Ext{}{Q}Mk$ as an $\Ext{}{R}kk$-module as well.  The map
$\Ext0{\vf}Mk_*$ is surjective, because it is canonically isomorphic to
the identity map of $\Hom k{M_0}k_*$.  It follows that $\Ext{}{\vf}Mk$
is surjective.
   \end{proof}

 \begin{chunk}
  \begin{proof}[Proof of Main Theorem]
Recall that $Q$ is Koszul, $J$ is a non-zero ideal of $Q$ with $J_1=0$,
and $c=\slope_R{R_{\irr}}$.  Note that $\vf$ is small by Proposition
\ref{prop:koszul_small}(1).

(1)  The inequality $\slope_QR\le c+1$ was proved as part of Theorem
\ref{thm:small}.
 
One has $\topp{i}Qk=i$ for $1\le i<\pd_Qk+1$ by the Koszul hypothesis
on $Q$, and $\topp{i}QR\ge i+1$ for $1\le i<\pd_QR+1$ by the conditions
$J_1=0$.  The exact sequence
  \[
\Tor{i+1}Qkk\to\Tor iQk{R_{\irr}}\to\Tor iQkR\,.
  \]
of graded vector spaces, which holds for every $i\ge1$, therefore implies
  \[
\topp iQ{R_{\irr}}\le\max\{\topp{i+1}Qk,\topp iQR\}=\topp iQR\,,
  \]
and hence $\slope_QR_{\irr}\le\sup_{i\ges1}\{(\topp iQR-1)/i\}$.
Now Proposition \ref{thm:ceub} gives
  \[
c \le\max\left\{\slope_Q{R_{\irr}},\sup_{i\ges1}\left\{\frac{\topp iQR-1}i\right\}\right\}
\le\sup_{i\ges1}\left\{\frac{\topp iQR-1}i\right\}
 \le\slope_QR\,.
  \]
When $\pd_QR$ is finite the last inequality is strict, so one has
$c<\slope_QR$.

The inequalities in (2), (3), and (4) were proved as part of Corollary
\ref{cor:small}.  

Finally, assume that $Q$ is a standard graded polynomial ring and
$\reg_QR=cp$ holds with $p=\pd_QR$.  Theorem \ref{thm:small} then shows
that $(\Tor1QkR_{c+1})^p$ is not zero, and so $\Ker\vf$ needs at least
$p$ minimal generators of degree $c+1$.  As a bigraded $k$-algebra,
$\Tor{}QkR$ is isomorphic to the homology of the Koszul $E$ complex on
some $K$-basis of $Q_1$, so one also has $(\HH1E)^p\ne0$.  Now a theorem
of Wiebe, see \cite[2.3.15]{BH}, implies that $\Ker\vf$ is generated by 
a $Q$-regular sequence of $p$ elements.
  \end{proof}
 \end{chunk}

  \begin{proposition}  
    \label{prop:canonical}
For a Koszul $K$-algebra $Q$ and $R=Q/J$ with $J\subseteq(Q_{\irr})^2$
one has
  \[
2\le\slope_QR\le\slope_{\wt R}R\,,
  \]
where $R=\wt R/I_R$ is the canonical presentation.
Equalities hold when $R$ is Koszul.
  \end{proposition} 
  
 \begin{proof}
The canonical presentation factors as ${\wt R}\to Q\xra{\vf}R$; see 
\ref{canonical}.   Part (1) of the main theorem, applied to the homomorphism
${\wt R}\to Q$ and the $Q$-module $R$, gives inequalities
$2\le\slope_{\wt R}Q\le\Rate Q+1=2$, so Proposition \ref{thm:ceub} yields
  \[
\slope_QR\le\max\{\slope_{\wt R}R,\slope_{\wt R}Q\}
=\max\{\slope_{\wt R}R,2\}=\slope_{\wt R}R\,.
  \]
When $R$ is Koszul, the computation above gives 
$2\le\slope_{\wt R}R\le\Rate R+1=2$.
 \end{proof}

The last assertion of Proposition \ref{prop:canonical} does not
admit a converse.  To demonstrate this we appeal to a family of 
graded algebras constructed by Roos \cite{Ro}.
Recall that the formal power series
$H_{M}(s)=\sum_{j\in\BN}\rank_KM_js^j$ in $\BZ[\![s]\!]$ is
called the \emph{Hilbert series} of $M$, and the formal Laurent
series $P^{R}_k(s,t)=\sum_{i\in\BN,j\in\BZ}\beta_{i,j}^R(M)\,s^jt^i $ in
$\BZ[s^{\pm1}][\![t]\!]$, where $\beta_{i,j}^R(M)=\rank_k\Tor iRkM_j$,
is known as its \emph{graded Poincar\'e series}.

  \begin{chunk}
    \label{ch:roos}
Let $P=K[x_1,x_2,x_3,x_4,x_5,x_6]$ be a polynomial ring.

For each integer $a\ge2$ set $R{(a)}=P/I{(a)}$, where $I{(a)}$ is the ideal
  \[
\big(\{x_i^2\}_{1\le i\le 6}\,,\,\{x_{i}x_{i+1}\}_{1\le i\le 5}\,,\,
x_1x_3+ax_3x_6-x_4x_6\,,\,x_1x_4+x_3x_6+(a-2)x_4x_6\big)\,.
  \]
When the characteristic of $K$ is zero, Roos \cite[Thm.\,1$'$]{Ro} proves
the equalities
  \[
H_{R(a)}(s)=1+6s+8s^2
  \quad\text{and}\quad
P^{R{(a)}}_k(s,t)=\frac1{H_{R(a)}(-st)-(st)^{a+1}(s+st)}\,.
  \]
  \end{chunk}

  \begin{example}
For each $a\ge2$ the graded $K$-algebra $R{(a)}$ from \ref{ch:roos} satisfies
  \[
\slope_{P}R{(a)}-1=1<1+(1/a)\le\Rate R{(a)}\le1+(2/a)\,.
  \]

Indeed, one has $\topp 1P{R(n)}=2$ because $I(a)$ is generated by
quadrics.  The isomorphism $\Tor iPk{R(a)}_j\simeq\HH i{E\otimes_P{R(a)}}_j$, 
where $E$ denotes the Koszul complex on some basis of
${P}_1$, and the equalities ${R(a)}_j=0$ for $j\ge3$ imply $\topp
iP{R(a)}\le i+2$ for $2\le i\le 6$.  Comparing the numbers $\topp
iP{R(a)}/i$, one gets $\slope_{P}R{(a)}=2$.

Following \cite{ABH}, for each
$f(s,t)=\sum_{i,j\ges0}b_{i,j}s^jt^i\in\BR[s][\![t]\!]$ we set 
  \[
\rate(f(s,t))=\sup_{i,j}\{j/i\mid i\ge1\text{ and }b_{i,j}\ne0\}\,.
  \]
Writing $h(s,t)=6-8st+s^{a+1}t^{a}+s^{a+1}t^{a+1}$, we obtain the
expression
  \[
P^{R(a)}_{R(a)_{\irr}}(s,t)=\frac{P^{R(a)}_k(s,t)-1}{t}
=\frac{sh(s,t)}{1-(st)h(s,t)}
=\sum_{i\ges1}s^it^{i-1}h(s,t)^i\,.
  \]
The momomial $s^jt^i$ with least $i\ge1$ and largest $j$, which appears
with a non-zero coefficient in the sum on the right, is $s^{a+2}t^{a}$.  This
gives the first inequality below:
  \begin{align*}
\frac{a+1}a
&\le\slope_{R(a)}(R(a)_{\irr})
=\rate\left(\frac{s\cdot h(s,t)}{1-(st)h(s,t)}\right)\\
&\le\max\big\{\rate(s\cdot h(s,t))\,,\rate(1-(st)h(s,t))\big\}\\
&=\max\left\{\frac{a+2}a\,,\frac{a+2}{a+1}\right\}
=\frac{a+2}a\,.
  \end{align*}
The second inequality comes from \cite[1.1]{ABH}.  The desired 
inequalities follow.
  \end{example}

\section{Slopes and  Gr\"obner bases}

Let $R$ be a standard graded $K$-algebra
and $R={\wt R}/I_R$ its canonical presentation.

Let $T(R)$ denote the set of all term orders on all $K$-bases of ${\wt R}_1$.
Letting $\ini_\tau(I_R)$ denote the initial ideal corresponding to $\tau\in T$, 
Eisenbud, Reeves, and Totaro \cite{ERT} set 
   \[
\Delta(R)=\inf_{\tau\in T(R)}\{ t^{\wt R}_1({\wt R}/\ini_\tau(I_R)) \}\,.
  \]
In words: $\Delta(R)$ is the smallest number $a$ such that $I_R$ has a 
Gr\"obner basis of elements of degree $\leq a$ with respect to a 
term order on some coordinate system.  Now we set
  \[
\Delta^{\ell}(R)=\inf\{ \Delta(Q) \}\,,
  \]
where $Q$ ranges over the set of all graded $K$-algebras satisfying 
$Q/L\simeq R$ for some ideal $L$ generated by a $Q$-regular sequence 
of elements of degree $1$.  

  \begin{proposition}  
   \label{prop:a1}
When $R$ is not a polynomial ring the following inequalities hold:
  \[
2\le\Rate R+1\le\Delta^{\ell}(R)\,.
    \qedhere
  \]
  \end{proposition} 
  
 \begin{proof}
For $R\cong Q/(l)$ with $l$ a non-zero-divisor in $Q_1$, one has  
a chain
  \[
\Rate R=\slope_{R}R_{\irr}=\slope_{Q}R_{\irr}
=\slope_{Q}Q_{\irr}=\Rate Q\le\Delta(Q)-1\,.
  \]
where the first and third equalities hold by definition, the second one 
by Proposition \ref{thm:reg}(1), and the last one from the exact 
sequence $0\to Q(-1)\to Q_{\irr}\to R_{\irr}\to 0$; the inequality, 
announced without proof by Backelin \cite[Claim, p.\,98]{Ba}, is 
established in \cite[Prop.~3]{ERT}.  The second inequality
in the proposition follows.
 \end{proof}

Combining the main theorem and the preceding proposition, one 
obtains:

  \begin{corollary}
    \label{cor:taylor}
The following inequalities hold.
 \begin{enumerate}[\rm(1)]
  \item
$\slope_{\wt R}R\le\Delta^{\ell}(R)$.
  \item
$t_i^{\wt R}(R)< \Delta^{\ell}(R)\cdot i$ for all $i>(\rank_KR_1-\dim R)$.
  \item
$\reg_{\wt R} R\leq(\Delta^{\ell}(R)-1)\cdot(\rank_KR_1-\depth R)$.
  \qed
  \end{enumerate}
  \end{corollary}

The research reported in this paper was prompted by the inequalities
above, which were initially obtained by a very different argument; we
proceed to sketch it.

  \begin{chunk}
    \label{rem:taylor}
For any isomorphism $R\simeq Q/L$, with $L$ generated by a regular
sequence of linear forms, and for each $\tau\in T(Q)$ and every pair of 
integers $(i,j)$ one has:
  \begin{equation}
    \label{eq:betti}
\beta_{i,j}^{\wt R}(R)=\beta_{i,j}^{\wt Q}(Q)\le
\beta_{i,j}^{\wt Q}(\wt Q/\ini_\tau(I_Q))\,;
  \end{equation}
see, for instance, \cite[3.13]{BC}.  The Taylor resolution 
of the monomial ideal $\ini_\tau(I_Q)$, see \cite[\S5]{Fr}, 
yields inequalities $\topp i{\wt Q}{{\wt Q}/\ini_\tau(I_Q)}\leq
\topp 1{\wt Q}{{\wt Q}/\ini_\tau(I_Q)}\cdot i$, which are strict for
$i>\rank_KQ_1-\dim Q$.  From these observations one obtains:
  \[
\slope_{\wt R}R=\slope_{\wt Q}{Q}=\sup_{i\ges1}\{\topp i{\wt Q}Q/i\} 
\le\inf_{\tau\in T(Q)}\{\topp 1{\wt Q}{{\wt Q}/\ini_\tau(I_Q)}\}=\Delta(Q)\,.
  \]
These inequalities imply part (1) of Corollary \ref{cor:taylor};
part (3) is a formal consequence.
  \end{chunk}

In \cite{Co}, algebras $R$ satisfying $\Delta(R)=2$ are called 
\emph{G-quadratic}, and those with $\Delta^{\ell}(R)=2$ are called 
\emph{LG-quadratic}.  A G-quadratic algebra is LG-quadratic by 
definition, and an LG-quadratic one is Koszul, see Proposition \ref{prop:a1}.

The first one of the preceding implications is not invertible:  By an
observation of Caviglia, see \cite[1.4]{Co}, complete intersections
of quadrics are LG-quadratic, while it is known that not all of them are 
G-quadratic, see \cite{ERT}.  Which leaves us with:

  \begin{question}
Is every Koszul algebra LG-quadratic?
  \end{question}

The \emph{Betti numbers} 
$\beta^{\wt R}_i(R)=\sum_{j\in\BZ}\rank_k\Tor i{\wt R}kR_j$ might help 
separate the two notions.  Indeed, when $R$ is LG-quadratic one has 
$R\cong Q/L$ and $Q=\wt Q/I_Q$, where $Q$ is a standard graded 
$K$-algebra, $L$ is an ideal generated by a $Q$-regular sequence of 
linear forms, and the initial ideal $\ini_\tau(I_Q)$ for some $\tau\in T(Q)$ 
is generated by quadrics.  As a consequence, one has
$\beta_1^{\wt Q}(Q)=\beta_1^{\wt Q}(\wt Q/\ini_\tau(I_Q))$, so we
obtain
 \begin{align*}
\beta_i^{\wt R}(R)
\le\beta_i^{\wt Q}(\wt Q/\ini_\tau(I_Q))
\le\binom{\beta^{\wt Q}_1(\wt Q/\ini_\tau(I_Q))}i
=\binom{\beta^{\wt R}_1(R)}i\,,
  \end{align*}
with inequalities coming from \eqref{eq:betti} and the Taylor resolution.  
Thus, we ask:

  \begin{question}
If $R$ is a Koszul algebra, does
$\beta^{\wt R}_i(R)\le\displaystyle\binom{\beta^{\wt R}_1(R)}i$ 
hold for every $i$?
  \end{question}

 \end{document}